\documentclass{article}
\usepackage[utf8]{inputenc}
\usepackage[T1]{fontenc}
\usepackage{geometry}
\author{David Forsman}
\title{Functors of Variance: Heuristic Naturality and Connection to Ends}
\date{\today}

\usepackage{graphicx}
\usepackage{amsthm, amssymb, mathtools}
\usepackage{enumitem}
\usepackage{hyperref}
\usepackage{lmodern}

\usepackage{tikz-cd}
\usepackage{caption, subcaption}
\usepackage[numbers]{natbib}
\usepackage{url}
\usepackage{microtype}

\theoremstyle{plain}
\newtheorem{theorem}{Theorem}[section]

\theoremstyle{definition}
\newtheorem{definition}[theorem]{Definition}
\newtheorem{example}[theorem]{Examples}

\newcommand{\ra}{\rightarrow}

\begin{document}

\maketitle
\begin{abstract}\noindent
The concept of a variance on a category is introduced as a two-sided strict factorization system. By employing variances, we define functors of variance in a more general setting than is usually considered, thereby eliminating the need for their domains to be product categories. Heuristic natural transformations are defined between functors of variance, whose domains are connected via a span. Heuristic naturality encompasses various known types of naturality, such as diagonal, extraordinary, and twisted naturality.

We demonstrate the connection between heuristic naturality and a generalized comma category, revealing the bijective correspondence between heuristic natural transformations and sections of the forgetful functor from the generalized comma category. Moreover, we introduce the concept of limit with respect to the notion of a heuristic transformation. This is the notion of end. It is shown that ends can be calculated via products and equalizers, and the Fubini theorem still holds for ends.
\end{abstract}

\section*{Introduction}
The inception of category theory was largely motivated by the notion of naturality, which can be viewed as a formalization of the concept of canonically constructed collections of functions. For instance, the collection of functions consisting of projections to the first component in a binary product exemplifies a natural transformation. However, some families of maps do not exhibit naturality in the conventional sense. Consider, for example, the linear isomorphisms between real Hilbert spaces and their duals. The coherence of this family of linear maps is encapsulated by twisted naturality.

Additionally, the family of evaluation maps constitutes a natural transformation in one variable and extra-natural in another. Diagonal naturality enables the recovery of these previously mentioned notions of naturality. Nonetheless, the approach to diagonal naturality tends to discard too much information about the original setting. In their paper \cite{Eilenberg1966}, Eilenberg and Kelly characterize when the vertical composition of extra-natural transformations is generally well-defined. Such a theorem proves challenging to recover from the generalized setting of dinatural transformations. Due to practical considerations, many mathematicians have preferred to explore extra-natural transformations over dinatural transformations in areas such as the theory of ends. The perspective of extra-naturality permits a suitable generalization to ends in the context of enriched categories.

In the first section, strict factorization systems are explored and, in particular, variances as two-sided strict factorization systems. A strict factorization system $(E,M)$ on a category $C$ is a pair of subcategories both containing all objects of $C$, where each morphism $f$ equals to $me$ for some unique morphisms $m$ in $M$ and $e$ in $E$. If $(M,E)$ is also a strict factorization system, then $(E,M)$ is said to be a variance on $C$. We show that any compatible pair of functors defined on a variance $(E,M)$ of $C$ uniquely determines a functor with $C$ as its domain. Furthermore, we define the notion of functor of variance. A functor $C\ra D$ of variance $(E,M)$ can be characterized as a compatible pair consisting of a covariant functor $E\ra D$ and a contravariant functor $M\ra D$. Compatibility is simply the requirement that the functors behave the same on objects and respect the factorizations of morphisms.

The second section will be dedicated to the exploration of heuristic naturality, a central concept in our categorical setting. The notion of a span in categories becomes an important way to relate two functors of variance with different domains. The general theory of heuristic transformations centers around spans, but in the special case where the domains are products of categories, we introduce a notion of partition of arguments. With the notion of partition of arguments, we are able to efficiently recover previous notions of naturality like diagonal naturality, extra-naturality, and twisted naturality. A partition of arguments can be read from a formal expression of labeled morphisms. This leads to the name  `heuristic naturality'.

Consider the expression
$$
F(x,x,x,y,z)\xrightarrow{f_{x,y,z}} G(x,z,z)
$$
of morphisms labeled with objects $x,y,z$ from their respective categories. For illustrative purposes, assume that $F$ and $G$ are functors, with $F$ being contravariant at positions $2$ and $5$, $G$ being contravariant at position $1$, and both being covariant in other positions. We define the heuristic naturality condition for the family of morphisms $(f_{x,y,z})_{x,y,z}$ as the commutativity of the diagram
$$
\begin{tikzcd}
{F(x,x,x,y,z)} \arrow[rr, "{f_{x,y,z}}"] & & {G(x,z,z)} \arrow[d, "{G(id,v,v)}"] \\
{F(x,x',x,y,z')} \arrow[u, "{F(id,s,id,id ,v)}"] \arrow[d, "{F(s,id,s,t,id)}"'] & & {G(x,z',z')} \\
{F(x',x',x',y',z')} \arrow[rr, "{f_{x',y',z'}}"'] & & {G(x',z',z')} \arrow[u, "{G(s,id,id)}"']
\end{tikzcd}
$$
for arbitrary morphisms $s\colon x\ra x'$, $t\colon y\ra y'$, and $v\colon z\ra z'$ in the respective categories. The diagram is essentially uniquely constructed from the requirement that there are two parallel paths containing the expressions $f_{x,y,z}$ and $f_{x',y',z'}$, respectively.

The third section brings our attention to the concept of comma categories, which play a crucial role in understanding the interplay between functors and heuristic natural transformations. A heuristic transformation can be seen as a section to a forgetful functor from a corresponding generalized comma category. This provides us with a new perspective on how componentwise naturality implying naturality relates to the lifting of functors.

Lastly, we will focus on the definitions and properties of ends and coends, which are intimately related to the notion of heuristic naturality. We demonstrate that in this generalized setting, the Fubini theorem of ends still holds and ends can be computed via equalizers and products when they exist.

\section{Functors of variance}
Consider a functor $F\colon C\times C\ra D$ which is contravariant in the first and covariant in the second argument. The goal is to talk about the functor $F$ without needing the domain of $F$ to be a product category. Instead, we specify a structure on the domain of $F$ that chooses coherently which morphisms are covariant and contravariant. We can then define that the functor $F$ acts covariantly on the covariant morphisms and contravariantly on the contravariant morphisms. The choice of covariant morphisms $E$ and contravariant morphisms $M$ will define a two-sided strict factorization system $(E,M)$. This structure allows us to produce a theory of functors of mixed variance.

\subsection{Strict factorization systems and variance}
Strict factorization systems were first introduced in the paper \cite{Grandis2000} by Grandis. We use strict factorization systems to formalize the notion of variance on a category and the notion of a functor of mixed variance. 
\begin{definition}
    Let $C$ be a category with subcategories $E$ and $M$. The pair $(E,M)$ is called a \textbf{strict factorization system} if
    \begin{itemize}
        \item the subcategories $E$ and $M$ are \textbf{wide}, meaning they contain all the objects of $C$;
        \item each morphism $f$ in $C$ has a unique factorization $f = me$ where $e$ and $m$ are morphisms in $E$ and $M$, respectively.
    \end{itemize}
    The pair $(E,M)$ is called a \textbf{variance} on $C$, if $(E,M)$ and $(M,E)$ are strict factorization systems on $C$. If $(E,M)$ is a variance on $C$, we call the morphisms in the subcategory $E$ \textbf{covariant morphisms} and those in $M$ \textbf{contravariant morphisms}. A covariant functor $C\ra D$ where $C$ and $D$ are categories with variance is said to preserve variance, if covariant and contravariant morphisms are mapped to covariant and contravariant morphisms, respectively. We denote by $\textbf{CatV}$ the category of small categories equipped with variance and variance-preserving functors among them. 
\end{definition}
The category \textbf{CatV} is a complete category with small coproducts. Through a direct verification, one notices that the forgetful functor $\textbf{CatV}\ra \textbf{Cat}$ strictly creates and preserves limits and coproducts. Recall that a functor $F\colon A\ra C$ is said to strictly create the limit of a diagram $D\colon I\ra A$ if each limiting cone $\theta'$ over $FD$ has a unique cone $\theta$ over $D$ that $F$ maps to $\theta'$; additionally, $\theta$ is a limiting cone. The functor $F$ is said to preserve the limit of $D$, if each limiting cone over $D$ is taken to a limiting cone over $FD$ via $F$. If $F$ preserves/strictly creates the limit for each diagram in $A$, then $F$ is said to preserve/strictly create limits, similarly for colimits.

Consider a category $C$ with a strict factorization system $(E,M)$. Notice that the subcategories $E$ and $M$ are closed under inversion. Generally, $E$ and $M$ satisfy the relative right and left cancellation laws, respectively. Since if $rs$ is a morphism in $E$, where $r$ is a morphism in $C$ and $s$ is a morphism in $E$, then $r = me$ for some morphisms $m$ in $M$ and $e$ in $E$, and so $rs = mes$. Thus $m$ has to be the identity, by unique decomposition. Thus the morphism $r = e$ is in $E$. Dually, the subcategory $M$ satisfies the relative left cancellation law.

\begin{definition}
    Let $C$ be a category with a variance $(E,M)$.
    Examine the unique factorizations of $f\colon x\ra y$ of a morphism in $C$
    $$
    \begin{tikzcd}
x \arrow[r, "f^e"] \arrow[d, "f_m"'] \arrow[rd, "f" description] & f_t \arrow[d, "f^m"] \\
f_s \arrow[r, "f_e"']                                          & y                 
\end{tikzcd}
    $$
    where $f^e,f_e$ and $f^m,f_m$ are morphisms of $E$ and $M$, respectively. We call the factorization $f = f^mf^e$ the \textbf{terminating factorization} of $f$ and $f = f_ef_m$ the \textbf{starting factorization} of $f$ with respect to the variance $(E,M)$ on $C$. The objects $f_s$ and $f_t$ are called the \textbf{starting} and \textbf{terminating objects}, respectively, associated to $f$.
\end{definition}
Remarkably, every strict factorization system on a groupoid can be viewed as a variance. To see this, consider a groupoid $G$ equipped with a strict factorization system $(E,M)$. Fix a morphism $g$ in $G$. Since $g^{-1} = me$ for some morphisms $m$ in $M$ and $e$ in $E$, we have $g = e^{-1}m^{-1}$, which yields an $(M,E)$ factorization of $g$. Since the intersection $E\cap M$ is a discrete category, this factorization is unique.

Turning to the case of a category $C$ with a variance $(E,M)$, suppose $f$ and $g$ are morphisms in $C$ with $f$ covariant and $g$ contravariant. In this setting, we have the equations:
\begin{equation*}
\begin{aligned}
    &f = f^e = f_e, \qquad &f^m = id_y, \qquad &f_m = id_x, \qquad &f_s = x, \qquad &f_t = y, \\
    &g = g^m = g_m, \qquad &g^e = id_x, \qquad &g_e = id_y, \qquad &g_s = y, \qquad &g_t = x.
\end{aligned}
\end{equation*}
Consider morphisms $x\xrightarrow{f}y\xrightarrow{g}z$ in a category $C$ with variance $(E,M)$. The morphisms $f$ and $g$ induce a commutative diagram
\begin{equation}\label{diagram induced by composable morphisms}
\begin{tikzcd}
x \arrow[rd, "f" description] \arrow[d, "f_m"'] \arrow[r, "f^e"] & f_t \arrow[d, "f^m" description] \arrow[r, "(g^ef^m)^e"]                    & (g^ef^m)_t = (gf)_t \arrow[d, "(g^ef^m)^m"] \\
f_s \arrow[r, "f_e"] \arrow[d, "(g_mf_e)_m"']                    & y \arrow[rd, "g" description] \arrow[r, "g^e"] \arrow[d, "g_m" description] & g_t \arrow[d, "g^m"]                        \\
(g_mf_e)_s = (gf)_s \arrow[r, "(g_mf_e)_e"']                     & g_s \arrow[r, "g_e"']                                                       & z                                          
\end{tikzcd}
\end{equation}
By the unique factorizations, we immediately obtain
\begin{align}\label{The variance factorization equations for composite}
    (gf)^m &= g^m(g^ef^m)^m
    &(gf)^e = (g^ef^m)^ef^e\\
    (gf)_m &= (g_mf_e)_m f_m
    &(gf)_e = g_e(g_mf_e)_e
\end{align}
Hence the codomain of $(gf)_m$ is the same as the codomain of $(g_mf_e)_m$, and so $(g_mf_e)_s = (gf)_s$. The equation on the terminating objects is obtained similarly: $(gf)_t = (g^ef^m)_t$. An important fact about the diagram (\ref{diagram induced by composable morphisms}) is that each morphism pointing to the right is covariant and each morphism pointing downwards is contravariant. Later we define what a functor of variance is. Such a functor maps a morphism $f\colon x\ra y$ to a morphism $F(f)\colon F(f_s)\ra F(f_t)$. Hence the contravariant morphisms have
their direction turned and the direction of a covariant morphism is respected. 
\begin{example}
The notion of variance is prevalent in many areas of mathematics:

\begin{itemize}
\item Every category $A$ has a \textbf{covariant variance}, where each morphism is covariant, and a \textbf{contravariant variance}, where each morphism is defined to be contravariant.\item A product category of categories with variance has a canonical variance structure. We define a morphism of the product category to be covariant (contravariant) if each projection is covariant (contravariant). This construction directly follows from the earlier observation that the forgetful functor $\textbf{CatV}\ra \textbf{Cat}$ strictly creates and preserves limits.

\item Let $N$ be a monoid considered as a single-object category. A variance on $N$ is a pair of submonoids $(E,M)$ where for each element $n$ of $N$, there exist unique elements $e, e' \in E$ and $m, m' \in M$ such that $n = me$ and $n = e'm'$. If $N$ is a group, then as previously noted for groupoids, any strict factorization system on $N$ will be a variance and both submonoids $E$ and $M$ will be subgroups.

\begin{itemize}
    \item Examine the set of strictly positive integers with the multiplicative structure. Take a partitioning of the prime numbers into $P_1, P_2$ and denote the generated submonoids by $\overline{P}_1$ and $\overline{P}_2$. By the fundamental theorem of arithmetic, $(\overline{P}_1, \overline{P}_2)$ is a variance on the multiplicative monoid of positive integers.
    
    \item Let $H$ and $N$ be groups. Assume $H$ acts on $N$. Consider the associated outer semi-direct product $H \ltimes N$. Here, $H$ and $N$ can be considered as subgroups of the outer semi-direct product. The pair $(N, H)$ defines a variance on $H \ltimes N$. The group $N$ is a normal subgroup of $H \ltimes N$.

    \item The notion of variance on a group is strictly more general than internal semi-direct products: Consider the symmetric group $S_4$ of permutations of four elements. The group $S_4$ has non-unique subgroups of order $8$ and $3$, respectively. Since $3$ and $8$ are coprime integers and $|S_4| = 3 \times 8$, it follows that all pairs $(E, M)$ form a variance on $S_4$ where $E$ and $M$ are subgroups of order $8$ and $3$, respectively. The fact that there are multiple subgroups of order $3$ and $8$ and the numbers are coprime implies that neither of the subgroups $E$ and $M$ can be normal.
\end{itemize}

\item Let $C$ be any category. Let $J$ be a union of path components of $C$. We may uniquely specify a variance on $C$ by setting all the morphisms whose domain lies in $J$ to be covariant and contravariant otherwise.

\item Let $I$ be a set. Let $A$ be a family of categories $A_i$ for $i\in I$, also called an $I$-\textbf{list} of categories. Let $v\colon I\ra \{0,1\}$ be a function. We call $v$ an \textbf{index-variance} on $A$ and an \textbf{index} $i\in I$ is said to be \textbf{covariant} if $v(i) = 0$ and \textbf{contravariant} otherwise. Consider the product category $\prod A = \prod_{i\in I} A_i$. The function $v$ defines a variance on $\prod A$ by setting $A_i$ to have covariant (contravariant) variance when $i$ is a covariant (contravariant) index for each $i\in I$. Thus, the product $\prod A$ obtains a variance from the categories $A_i$, for $i\in I$, via the product. We call this variance on $\prod A$ an index-variance with respect to $v$.
\begin{itemize}
    \item If $A_1,\ldots, A_n$ are categories and $s_1,\ldots, s_n\in\{0,1\}$, we say that the tuple $v = (s_1,\ldots, s_n)$ defines an index-variance on the category 
    $$
    A_1\times \cdots\times A_n.
    $$
\end{itemize}

\item Let $C$ be a category with variance $(E,M)$. A subcategory $R$ of $C$ inherits the variance of $C$ if $R$ is closed under factoring with respect to the variance: If $f$ is a morphism in $R$, then $f^m,f^e,f_m,f_e$ are morphisms in $R$. Then the restrictions of $E$ and $M$ to $R$ define a variance on $R$.

\begin{itemize}
    \item Consider an $I$-list of categories with variance $A$. The product category $\prod A$ has a subcategory of \textbf{finitely supported} morphisms. A morphism $f$ of $\prod A$ is finitely supported if $f_i$ is not an identity morphism for at most finitely many $i\in I$. The subcategory of finitely supported morphisms of $\prod A$ inherits the variance from $\prod A$.
\end{itemize}\end{itemize}
\end{example}

One important property of a variance on a category is that a functor with a domain equipped with a variance can be reconstructed from a compatible pair of functors defined only on the contravariant and covariant morphisms. 
\begin{theorem}[Decomposition theorem]\label{Decomposition theorem}
    Let $C$ be a category with variance $(E,M)$. Assume that $G\colon E\ra D$ and $H\colon M\ra D$ are compatible functors: The functors $G$ and $H$ map objects the same and
    $$
    H(f^m)G(f^e) = G(f_e)H(f_m)
    $$
    for each morphism $f$ in $C$. Then there exists a unique functor $F\colon C\ra D$ extending $G$ and $H$. 
\end{theorem}
\begin{proof}
    The uniqueness of $F$ is clear. We show existence. Define a functor $F\colon C\ra D$ mapping the objects the same as $G$. For a morphism $f$ in $C$ we define that $F(f) = H(f^m)G(f^e)$. Identities are taken to identities by $F$. Left to show is the compatibility with composition. Let $x\xrightarrow{f}y\xrightarrow{g} z$ be morphisms in $C$. Now
    \begin{align*}
        F(gf) 
        &= H((gf)^m)G((gf)^e)\\
        &= H(g^m(g^ef^m)^m) G((g^ef^m)^e f^e)\\
        &= H(g^m)H((g^ef^m)^m) G((g^ef^m)^e) G(f^e)\\
        &= H(g^m)G((g^ef^m)_e)H((g^ef^m)_m)G(f^e)\\
        &= H(g^m)G(g^e) H(f^m)G(f^e)\\
        &= F(g)F(f) & \tag*{\qedhere}
    \end{align*}
\end{proof}
Let $C$ be a category with a variance $(E,M)$. The subcategories $E$ and $M$ define the canonical functor $E+M\ra C$. Let $D$ be the full subcategory of the coslice category $(E+M)/\textbf{Cat}$ consisting of functors whose generalized kernel, as defined in \cite{Bednarczyk1999}, contains the pair of sequences of morphisms $((f^m,f^e),(f_e,f_m))$ for every morphism $f$ in $C$. The previous theorem demonstrates that the functor $E+M\ra C$ is initial in $D$.

Now, given a pair $(A,B)$ of subcategories of $C$, we ask whether we can define a subcategory of $(A+B)/C$ such that if the canonical functor $A+B\ra C$ is an initial object of this subcategory, then the pair $(A,B)$ defines a variance on $C$. An affirmative answer to this question could allow us to generalize the notion of variance from the category $\textbf{Cat}$ of small categories to a broader class of categories.

\subsection{Functors of mixed variance}
Consider the Decomposition Theorem \ref{Decomposition theorem} where a covariant functor $A\ra C$ is reconstructed from a pair of compatible covariant functors $E\ra C$ and $M\ra C$ defined on a variance $(E,M)$ on the category $A$. In this subsection, we examine the possibility that the functor $M\ra C$ is instead a contravariant functor. With a notion of compatibility, we characterize the notion of a functor of mixed variance as a compatible pair of functors, consisting of a covariant functor $E\ra C$ and a contravariant functor $M\ra C$.
\begin{definition}[Functor of mixed variance]
Let $A$ and $C$ be categories and let $(E,M)$ be a variance on $A$. A functor $A\xrightarrow[]{F} C$ of variance $(E,M)$ consists of a pair of functions from the objects and morphisms of $A$ to the objects and morphisms of $C$, respectively, both denoted by $F$:
\begin{itemize}
    \item Identities are preserved under $F$.
    \item Let $f\colon x\ra y$ be an morphism of $A$. Then $F(f)\colon F(f_s)\ra F(f_t)$, where $f_s$ and $f_t$ are the starting and ending objects associated to $f$ with respect to the variance $(E,M)$.
\end{itemize}
Furthermore, we require the \textbf{compatibility with composition}: Let $x\xrightarrow{f}y\xrightarrow{g}z$ be morphisms in $A$. The diagram
$$
\begin{tikzcd}
F(f_s) \arrow[rr, "F(f)"]                                                          &  & F(f_t) \arrow[d, "F((g^ef^m)^e)"]  \\
F(gf)_s \arrow[u, "F((g_mf_e)_m)"] \arrow[d, "F((g_mf_e)_e)"'] \arrow[rr, "F(gf)"] &  & F(gf)_t                            \\
F(g_s) \arrow[rr, "F(g)"']                                                         &  & F(g_t) \arrow[u, "F((g^ef^m)^m)"']
\end{tikzcd}
$$
commutes. We record the equations:
\begin{equation}\label{Functoriality equations}
\begin{aligned}
    &F(gf) = F((g^ef^m)^e) F(f) F((g_mf_e)_m)\\
     &F(gf) = F((g^ef^m)^m) F(g) F((g_mf_e)_e)
\end{aligned}
\end{equation}

\end{definition}

Consider a category $A$ with a variance $(E,M)$. Let $x\xrightarrow{f}y\xrightarrow{g}z$ be morphisms in $A$. We obtain the commutative diagram
$$
\begin{tikzcd}
x \arrow[rd, "f" description] \arrow[d, "f_m"'] \arrow[r, "f^e"] & f_t \arrow[d, "f^m" description] \arrow[r]                    & (gf)_t \arrow[d] \\
f_s \arrow[r, "f_e"] \arrow[d]                    & y \arrow[rd, "g" description] \arrow[r, "g^e"] \arrow[d, "g_m" description] & g_t \arrow[d, "g^m"]                        \\
(gf)_s \arrow[r]                     & g_s \arrow[r, "g_e"']                                                       & z                                          
\end{tikzcd}
$$
This diagram is the same as (\ref{diagram induced by composable morphisms}). Notice again that all vertical and horizontal arrows are contravariant and covariant, respectively. Applying a functor $F\colon A\ra C$ of variance $(E,M)$, we obtain the diagram 
$$
\begin{tikzcd}
F(x) \arrow[r]                                            & F(f_t) \arrow[r]                                          & F((gf)_t)        \\
F(f_s) \arrow[u] \arrow[r] \arrow[ru, "F(f)" description] & F(y) \arrow[u] \arrow[r]                                  & F(g_t) \arrow[u] \\
F((gf)_s) \arrow[u] \arrow[r]                             & F(g_s) \arrow[r] \arrow[u] \arrow[ru, "F(g)" description] & F(z) \arrow[u]  
\end{tikzcd}
$$
Erasing the objects $F(x),F(y)$ and $F(z)$ from the diagram with the associated morphisms and adding the morphism $F(gf)$ yields the diagram
$$
\begin{tikzcd}
                                                                & F(f_t) \arrow[r]                      & F((gf)_t)        \\
F(f_s) \arrow[ru, "F(f)" description]                           &                                       & F(g_t) \arrow[u] \\
F((gf)_s) \arrow[u] \arrow[r] \arrow[rruu, "F(gf)" description] & F(g_s) \arrow[ru, "F(g)" description] &                 
\end{tikzcd}
$$
which is exactly the diagram for the compatibility with composition.

Similarly, as before, we obtain a decomposition theorem, but this time for functors of mixed variance. 
\begin{definition}[Compatible Pair of Functors]
Given a category $A$ with variance $(E, M)$ and a category $C$, a \textbf{compatible pair} of functors consists of a covariant functor $G\colon E\ra C$ and a contravariant functor $H\colon M\ra C$, satisfying the following conditions:
\begin{itemize}
\item $G$ and $H$ map objects the same;
\item $G(f^e)H(f_m) = H(f^m)G(f_e)$ holds for every morphism $f$ in $A$.
\end{itemize}
\end{definition}

\begin{theorem}[Decomposition theorem of mixed variance]
Let $A$ be a category with variance $(E,M)$, and let $C$ be a category. Consider a functor $F\colon A\ra C$ of variance $(E,M)$. We assign $F\mapsto (F|_M E,F|_ M)$. This assignment is bijective on compatible pairs $(G\colon E\ra C, H\colon M\ra C)$.
\end{theorem}
\begin{proof}
    First, we show that the assignment is well defined. For this, it suffices to show that the functor $F$ behaves covariantly on $E$, contravariantly on $M$, and that the diagram
    $$
    \begin{tikzcd}
F(x) \arrow[r, "F(f^e)"]                                                      & F(f_t)                    \\
F(f_s) \arrow[r, "F(f_e)"'] \arrow[u, "F(f_m)"] \arrow[ru, "F(f)" description] & F(y) \arrow[u, "F(f^m)"']
\end{tikzcd}
    $$
    commutes for each morphism $f\colon x\ra y$ in $A$. Assume first that $x\xrightarrow{f}y\xrightarrow{g}y$ are of covariant type. Notice that $f^e = f_e = f$ and $f^m = id_y$ and $f_m = id_x$. Thus
    \begin{align*}
    F(gf)
    &= F((g^ef^m)^e)F(f)F((g_mf_e)_m)\\
    &= F(g) F(f) F(id) \\
    &= F(g)F(f).
    \end{align*}
    Assume next that $g$ and $f$ are of contravariant type. Notice that $f^m =f_m = f$ and $f^e = id_x$ and $f_e = id_y$. Now
    \begin{align*}
    F(gf) 
    &= F((g^ef^m)^e) F(f) F((g_mf_e)_m)\\
    &= F(id)F(f)F(g)\\
    &= F(f)F(g)
    \end{align*}
Let $f$ be an arbitrary morphism of $A$. Using the first of (\ref{Functoriality equations}) we have
    \begin{align*}
        F(f) 
        &=F(fid_x)\\
        &=F((f^eid_x^m)^e)F(id)F((f_mid_{x,e})_m\\
        &=F(f^e)F(f_m)
    \end{align*}
    and using the second equation of (\ref{Functoriality equations})
    \begin{align*}
        F(f)
        &= F(id_y f)\\
        &= F((id_y^ef^m)^m)F(id_y)F((id_{y,m}f_e)_e)\\
        &= F(f^m) F(f_e).
    \end{align*}
    Thus the assignment $F\mapsto (F\mid_ E,F\mid_ M)$ is well defined onto compatible pair of functors. 

    Fix a compatible pair $(G\colon E\ra C,H\colon M\ra C)$ of functors. Define a functor $F\colon A\ra C$ of variance $(E,M)$ mapping the objects the same as $G$ and $H$. For a morphism $f\colon x\ra y$ define $F(f) = G(f^e)H(f_m)$. Now $F(f)\colon F(f_s)\ra F(f_t)$. Evidently, $F$ maps identities to identities. To show the functoriality conditions, fix morphism $x\xrightarrow{f}y\xrightarrow{g}z$. Now
    \begin{align*}
        F(gf) 
        &= G((gf)^e)H((gf)_m) &\text{(The definition of $F$)}\\
        &= G((g^ef^m)^ef^e) H((g_mf_e)_mf_m) & (\ref{The variance factorization equations for composite}) \\
        &= G((g^ef^m)^e)G(f^e)H(f_m)H((g_mf_e)_m)&\text{(The functoriality of $G$ and $H$)}\\
        &= F((g^ef^m)^e)F(f)F((g_mf_e)_m) &\text{(The definition of $F$)}
    \end{align*}
    and using the compatibility condition we attain
    \begin{align*}
    F(gf)
    &= G((gf)^e)H((gf)_m) &\text{(The definition of $F$)} \\
    &= H((gf)^m)G((gf)_e) &\text{(Compatibility condition)} \\
    &= H(g^m(g^ef^m)^m)G(g_e(g_mf_e)_e) &(\ref{The variance factorization equations for composite}) \\
    &= H((g^ef^m)^m)H(g^m)G(g_e)G((g_mf_e)_e) &\text{(The functoriality of $G$ and $H$)} \\
    &= H((g^ef^m)^m)G(g^e)H(g_m)G((g_mf_e)_e) &\text{(Compatibility condition)} \\
    &= F((g^ef^m)^m)F(g)F((g_mf_e)_e) &\text{(The definition of $F$).}
\end{align*}

    Lastly, the functor $F$ of variance $(E,M)$ is the unique functor that the assignment maps to $(G,H)$. This proves the claim.
\end{proof}

A compatible pair of functors defined on covariant and contravariant morphisms, respectively, gives rise to a functor of mixed variance from the whole domain. This makes it easier to conceptualize functors of mixed variance. 
\begin{example}\hfill
    \begin{itemize}
        \item A covariant functor $C\ra D$ is a functor of mixed variance with respect to the covariant variance on $C$. Similarly for contravariant functors from $C$.
        \item Let $C$ be a locally small category. The hom-functor is a functor $C\times C\ra \textbf{Set}$ of index-variance $(1,0)$.
        \item Let $C$ be a category with variance $(E,M)$. Let $F\colon C\ra D$ be a functor. Assume that $F(m)$ is invertible for each morphism $m$ in $M$. We can define a functor $G\colon C\ra D$ of variance $(E,M)$ mapping objects the same as $F$. We set $G(e) = F(e)$ and $G(m) = F(m)^{-1}$ for all morphism $e$ and $m$ in $E$ and $M$, respectively. To show that we have a well-defined functor of variance, it suffices to check the compatibility condition. Let $f\colon x\ra y$ be a morphism in $C$. We need to check that $F(f^e)F(f_m)^{-1} = F(f^m)^{-1}F(f_e)$, which is equivalent to 
        $F(f^m)F(f^e) = F(f_e) F(f_m)$, which by functoriality is just the equality $F(f) = F(f)$.

        Essentially the same construction works the other way around too. If a functor of mixed variance is given and all the contravariant morphisms are mapped to invertible ones, we may turn the functor of mixed variance into a covariant functor.
        
        \item Let $G$ be a group with a variance $(E,M)$. The mapping $g\mapsto g^e$ is a group homomorphism $G\ra E$ if and only if $g^e = g_e$ for every $g\in G$. This map restricts into two trivial group homomorphisms, $E\xrightarrow{id_E} E$ and $M\xrightarrow{0} E$, and the compatibility condition is exactly $g^e = g_e$ for each $g\in G$. If $E$ is a normal subgroup of $G$, then this condition is equivalent to the elements of $E$ commuting with the elements of $M$, which forces an isomorphism $G\cong E\times M, g\mapsto (g^e,g^m)$.
    \end{itemize}
\end{example}
\section{Heuristic naturality}
The concept of a functor with variance allows us to define heuristic natural transformations, which encompass the usual notions of naturality like dinaturality and extraordinary naturality.

\begin{definition}[Spans]
    Let $A$ and $B$ be categories. A \textbf{span} on $A$ is a pair $(R,L)$, where $R$ is a category and $L\colon R\ra A$ is a functor. A span $(R,L)$ from $A$ to $B$ is a span on $A\times B$ and we denote this by $A\ra B$. By $L_i$ we denote the functor $L$ post-composed with the projection $\pi_i$ for $i = 1,2$.
\end{definition}

\begin{definition}[Heuristic naturality]
    Let $A,B$ and $R$ be categories. Assume that $A$ and $B$ are associated with variances. Let $F\colon A\ra C$ and $G\colon B\ra C$ be functors of variance. Let $(R,L)$ be a span $A\ra B$. Let $\eta$ be a family of morphisms 
    $$
    \eta_x\colon FL_1(x)\ra GL_2(x)
    $$
    for each object $x$ of $R$. We call $\eta\colon F\xRightarrow[L]{} G$ a \textbf{heuristic transformation} and an $L$-\textbf{transformation}, if $\eta$ satisfies the \textbf{heuristic naturality condition}: Let $f\colon x\ra y$ be a morphism in $R$. Consider the factorizations of $g\coloneqq L_1(f)$ and $h \coloneqq L_2(f)$
    $$
    \begin{tikzcd}
L_1(x) \arrow[d, "g_m"] \arrow[rd, "g" description] &        & L_2(x) \arrow[r, "h^e"] \arrow[rd, "h" description] & h_t \arrow[d, "h^m"] \\
g_s \arrow[r, "g_c"]                                & L_1(y) &                                                     & L_2(y)              
\end{tikzcd}
    $$
    The $L$-naturality with respect to $f$ is the requirement that the diagram
    $$
\begin{tikzcd}
F(L_1(x)) \arrow[rr, "\eta_x"]                  &  & G(L_2(x)) \arrow[d, "G(h^e)"]  \\
F(g_s) \arrow[u, "F(g_m)"] \arrow[d, "F(g_e)"'] &  & G(h_t)                         \\
F(L_1(y)) \arrow[rr, "\eta_y"']                 &  & G(L_2(y)) \arrow[u, "G(h_m)"']
\end{tikzcd}
    $$
    commutes. The $L$-naturality of $\eta$ means that $\eta$ is $L$–natural with respect to all morphisms $f$ in $R$.
\end{definition}

The notions of dinaturality and extra-naturality involve the domains of the functors to be product categories that have possibly the same category appearing more than once in the domains. We define a relation called partitioning of arguments that records which categories will be thought the same.
\begin{definition}[Partition of arguments]
    Let $A$ be an $I$-list of categories. Let $R$ be an equivalence relation on $I$ satisfying the \textbf{compatibility condition} where $iRj$ implies $A_i = A_j$ for each $i,j\in I$. The relation $R$ is called the \textbf{partitioning of the arguments} on $A$. If, furthermore, $B$ is a $J$-list of categories we say that a relation $S\colon A\ra B$ is a partition of arguments from $A$ to $B$, if $S$ is a partition of arguments of the \textbf{concatenated} $(I+J)$-list of categories $AB$.\footnote{An $I$-list of categories $A$ can be thought as a functor $I\ra\textbf{CAT}$, where the set $I$ is identified with a small discrete category. The concatenation of functors $I\ra\textbf{CAT}$ and $J\ra\textbf{CAT}$ is the canonical functor $(I+J)\ra\textbf{CAT}$, via the universal property of coproducts.}
\end{definition}

Let $A$ be an $I$-list of categories. Let $R$ be a partitioning of arguments of $A$. The relation $R$ produces a span on $\prod A$. We define the category $A_R$ as subcategory of $\prod A$ by choosing the morphisms $f$, called $R$-morphisms, of $\prod A$ satisfying $f_i = f_j$ for all $(i,j)\in R$. Similarly, a partition of arguments $S\colon A\ra B$ produces a span $\prod A\ra \prod B$ which we also denote with $S$. 

Consider functors $F\colon \prod A\ra C$ and $G\colon \prod B\ra C$ of index-variances $s$ and $t$, respectively. Let $R\colon A\ra B$ be a partition of the arguments. An $R$-transformation $\eta\colon F\Rightarrow G$ consists of morphisms
$$
\eta_x\colon F(\pi_1 x)\ra G(\pi_2 x)
$$
for $R$-objects $x$. Moreover, they must satisfy the following condition: Let $f\colon x\ra y$ be an $R$-morphism. Consider the factorizations of $f$ in $\prod AB$:
$$
\begin{tikzcd}
x \arrow[r, "f^e"] \arrow[d, "f_m"'] \arrow[rd, "f" description] & f_t \arrow[d, "f^m"] \\
f_s \arrow[r, "f_e"']                                            & y                   
\end{tikzcd}
$$
The heuristic naturality condition with respect to $f$ is then the commutativity of the diagram
$$
\begin{tikzcd}
F(\pi_1 x) \arrow[rr, "\eta_x"]                                  &  & G(\pi_2 x) \arrow[d, "G(\pi_2 f^e)"] \\
F(\pi_1 f_s) \arrow[u, "F(\pi_1f_m)"] \arrow[d, "F(\pi_1 f_e)"'] &  & G(\pi_2 f_t)                         \\
F(\pi_1y) \arrow[rr, "\eta_y"']                                  &  & G(\pi_2y) \arrow[u, "G(\pi_2 f^m)"']
\end{tikzcd}
$$
It is worth noting that the $R$-category $AB_R$ is isomorphic to $\prod_{i\in K} (AB)_i$, where $K$ is a set of representatives of the equivalence relation $R$. We view the pre-ordered set $(I+J,R)$ as a category and define a functor $\overline{AB}\colon (I+J,R)\ra\textbf{Cat}$ that maps the elements of $I+J$ to the same as $AB$ and maps all morphisms to identities. We see that $(AB)_R$ is the limit of the functor $\overline{AB}$ and the inclusion $K\hookrightarrow (I+J,R)$ is an equivalence of categories. This shows that $(AB)_R$ is isomorphic to the product $\prod_{i\in K} (AB)_i$.

For instance, in practice we may encounter a generic expression of the type:
\begin{equation}\label{Formal expression for natural maps}
    F(x,y,y)\xrightarrow{f_{x,y}}G(x,x,y).
\end{equation}
Here we assume that $F\colon X\times Y\times Y\ra C$ and $G\colon X\times X\times Y\ra C$ of variances $(1,0,1)$ and $(1,1,0)$, respectively. We can deduce the naturality condition for morphisms $f_{x, y}$ from the variances of $F$ and $G$. Observe that the partitioning of the arguments $R$ can be read from the variables used in the expression. We identify those positions, where the same variable appears. Here $R$ has two equivalence classes $\{1,4,5\}$ and $\{2,3,6\}$, since $x$ appears in positions $1,4$ and $5$, and $y$ appears in positions $2,3$ and $6$. Fixing morphisms $\alpha\colon x\ra x'$ and $\beta\colon y\ra y'$, we can see that the $R$-naturality condition with respect to $\alpha$ and $\beta$ is then the commutativity of the diagram
$$
\begin{tikzcd}
{F(x,y,y)} \arrow[rr, "{f_{x,y}}"]                                                         &  & {G(x,y,y)} \arrow[d, "{G(id_x,id_y,\beta)}"]          \\
{F(x',y,y')} \arrow[u, "{F(\alpha, id_y,\beta)}"] \arrow[d, "{F(id_{x'},\beta,id_{y'})}"'] &  & {G(x,y,y')}                                           \\
{F(x',y',y')} \arrow[rr, "{f_{x',y'}}"']                                                   &  & {G(x',y',y')} \arrow[u, "{G(\alpha,\beta,id_{y'})}"']
\end{tikzcd}
$$
Since the formal expression (\ref{Formal expression for natural maps}) yields the partition of arguments $R$, we may leave $R$ implicit in our conversation about heuristic naturality, if the generic expression is already understood. This explains the word `heuristic' in the name `heuristic naturality'.
\begin{example}
We attain the notions of twisted, diagonal and extraordinary naturality as special cases. Furthermore, the naturality of the family of evaluations is obtained.
\begin{description}
    \item[Naturality] Let $F\colon C\ra D$ and $G\colon C\ra D$ be covariant functors. Let $\eta$ be a collection of morphisms $\eta_x\colon Fx\ra Gx$ indexed over the objects of $C$. The collection $\eta$ defines a natural transformation if and only if it defines a $\Delta$-transformation where $\Delta\colon C\ra C\times C$ is the diagonal functor.

    \item [Twisted naturality] Let $F,G\colon C\ra D$ be functors where $F$ is covariant and $G$ contravariant. Let $\eta$ be a collection of morphisms $\eta_x\colon Fx\ra Gx$ indexed over the objects of $C$. The $\Delta$-naturality for $\eta$ is that the diagram
    $$
    \begin{tikzcd}
Fx \arrow[d, "F(f)"'] \arrow[r, "\eta_x"] & Gx                    \\
Fy \arrow[r, "\eta_y"]                    & Gy \arrow[u, "G(f)"']
\end{tikzcd}
    $$
    commutes for any morphism $f\colon x\ra y$ in $C$.

    \item [Diagonal naturality] Let $F,G\colon C\times C\ra D$ be functors with variance $(1,0)$. Let $\eta$ be a collection of morphisms $\eta_x\colon F(x,x)\ra G(x,x)$, then the $(\Delta,\Delta)$-naturality condition for $\eta$ (or $R$-naturality for the appropriate partition of arguments $R$) is the commutation of the diagram
    $$
    \begin{tikzcd}
{F(x,x)} \arrow[r, "\eta_x"]                              & {G(x,x)} \arrow[d, "{G(id_x,f)}"]  \\
{F(y,x)} \arrow[u, "{F(f,id_x)}"] \arrow[d, "{F(id,f)}"'] & {G(x,y)}                           \\
{F(y,y)} \arrow[r, "\eta_y"']                             & {G(y,y)} \arrow[u, "{G(f,id_x)}"']
\end{tikzcd}
    $$
    for any morphism $f\colon x\ra y$ in $C$. This is exactly the condition for diagonal naturality. It is of interest to note that every heuristic transformation between functors of index-variance can be thought of as a dinatural transformation:
    Let $A$ be an $I$-list of categories and let $B$ be a $J$-list of categories. Given a span $(R,L)\colon \prod A\ra \prod B$ and functors of index variance $F\colon \prod A\ra C$ and $G\colon \prod B\ra C$, we define functors $\overline{F},\overline{G}\colon R\times R\ra C$ with index-variance $(1,0)$ as follows:
For each index $i \in I+J$, if $i$ is a covariant index, we consider the functor $R\times R\xrightarrow{\pi_2} R\xrightarrow{\pi_i\circ L} (AB)_i$. If $i$ is a contravariant index, we consider the functor $R\times R\xrightarrow{\pi_1} R\xrightarrow{\pi_i\circ L} (AB)_i$. As $i$ runs over the set $I+J$, the considered functors, by the universal property of the product, yield functors $R\times R\ra \prod A$ and $R\times R\ra \prod B$. Post-composing these functors with $F$ and $G$, respectively, yields the functors of variance $\overline{F}$ and $\overline{G}$.

Now, consider the family $\eta$ of morphisms
    $
    \eta_x\colon \overline{F}(x,x)\ra \overline{G}(x,x)
    $
for each object $x$ of $R$. After a direct verification, we find that $\eta$ defines a dinatural transformation $\overline{F}\Rightarrow\overline{G}$ if and only if the family of morphisms $\eta$ defines an $L$-transformation $F\Rightarrow G$.
 
    \item [Extra-naturality] Let $F\colon A\times B\times B\ra D$ and $G\colon A\times C\times C\ra D$ be functors of variance $(0,1,0)$. The heuristic naturality condition for a collection $\eta$ of morphisms 
    $
    \eta_{x,y,z}\colon F(x,y,y)\ra G(x,z,z),
    $
    indexed over objects $x,y,z$ of $A,B$ and  $C$, respectively, is that the diagram
    $$
    \begin{tikzcd}
{F(x,y,y)} \arrow[rr, "{\eta_{x,y,z}}"]                                &  & {G(x,z,z)} \arrow[d, "{G(f,id_z,h)}"]              \\
{F(x,y',y)} \arrow[u, "{F(id_x,g,id_y)}"] \arrow[d, "{F(f,id_y',g)}"'] &  & {G(x',z,z')}                                       \\
{F(x',y',y')} \arrow[rr, "{\eta_{x',y',z'}}"']                         &  & {G(x',z',z')} \arrow[u, "{G(id_{x'},h,id_{z'})}"']
\end{tikzcd}
    $$
    commutes for morphisms $f\colon x\ra x'$, $g\colon y\ra y'$ and $h\colon z\ra z'$ in the categories $A,B$ and $C$, respectively. This notion of heuristic naturality is equivalent to extra-naturality. However, the usual condensed diagram for the extra-naturality condition is given by the diagram
    $$
    \begin{tikzcd}
{F(x,y,y)} \arrow[rr, "{\eta_{x,y,z'}}"]                               &  & {G(x,z',z')} \arrow[d, "{G(f,h,id_{z'})}"]    \\
{F(x,y',y)} \arrow[u, "{F(id_x,g,id_y)}"] \arrow[d, "{F(f,id_y',g)}"'] &  & {G(x',z,z')}                                  \\
{F(x',y',y')} \arrow[rr, "{\eta_{x',y',z}}"']                          &  & {G(x',z,z)} \arrow[u, "{G(id_{x'},id_z,h)}"']
\end{tikzcd}
    $$
    for morphisms $f,g,h$ (see \cite{Eilenberg1966}). This diagram does not allow an immediate generalization, which could be one of the reasons heuristic naturality has not been discovered before. 

    \item [Naturality of evaluation]
    Let $C$ be a monoidal closed category with the monoidal product and internal hom-functors $\otimes, [-,-]\colon C\times C\ra C$ with variances $(0,0)$ and $(1,0)$, respectively. Denote the counit of the adjunction as a family of morphisms
    $$
    ev_{a,b}\colon [a,b]\otimes a\ra b\text{ for objects $a,b$ in $C$}.
    $$
    From the formal expression itself, we may guess, what kind of heuristic naturality condition $ev$ satisfies. Let $f\colon a\ra a'$ and $g\colon b\ra b'$ be morphisms in $C$. Then the naturality condition is that the diagram
    $$
    \begin{tikzcd}
{[a,b]\otimes a} \arrow[rr, "{ev_{a,b}}"]                                         &  & b \arrow[dd, "g"] \\
{[a',b]\otimes a} \arrow[d, "{[id,g]\otimes f}"'] \arrow[u, "{[f,id]\otimes id}"] &  &                   \\
{[a',b']\otimes a'} \arrow[rr, "{ev_{a',b'}}"']                                   &  & b'               
\end{tikzcd}
    $$
    commutes. This is the correct notion of naturality for the evaluation maps.
\end{description}
\end{example}
The family of evaluation maps of a monoidal closed category is a motivating example for the notion of generalized extra-naturality, which is defined in the paper \cite{Eilenberg1966} by Eilenberg and Kelly. In our formulation, a general extra-natural transformation between functors of index variance $F\colon \prod A\ra C$ and $G\colon \prod B\ra C$ consists of an $R$-transformation $\eta\colon F\Rightarrow G$ where each equivalence class of the partition of arguments $R$ has exactly two elements. Additionally, if $i R j$ and $i\neq j$, then $i$ and $j$ differ in variance, if and only if they are elements of the same set $I$ or $J$ in $I+J$.

\section{Generalized comma categories}
We define the generalized comma category, which is connected to the definition of heuristic naturality. The generalized comma category gives us a conceptual framework to view heuristic natural transformation as sections to a forgetful functor. This is used to prove that for naturality, it suffices to check naturality with respect to morphisms of a certain generating subgraph. This generalizes the case where componentwise naturality implies naturality when the domain of the functors is a finite product of categories.
\begin{definition}[Generalized comma category]
    Let $F\colon A\ra C$ and $G\colon B\ra C$ be functors with variance. Let $(R,L)\colon A\ra B$ be a span. We define the generalized comma category $F\downarrow_L G$ as follows:
    \begin{itemize}
        \item The objects are pairs $(x,\alpha\colon FL_1 x\ra GL_2 x)$ where $x$ is an object of $R$ and $\alpha$ is a morphism in $C$.
        \item A morphism $(x,\alpha)\ra (y,\beta)$ consists of a morphism $f\colon x\ra y$ in $R$ making the diagram
        $$
\begin{tikzcd}
FL_1x \arrow[rr, "\alpha"]                                        &  & GL_2x \arrow[d, "G(\pi_2g^e)"]  \\
F(\pi_1 g_s) \arrow[u, "F(\pi_1 g_m)"] \arrow[d, "F(\pi_1 g_e)"'] &  & G(\pi_2 g_t)                    \\
FL_1y \arrow[rr, "\beta"']                                        &  & GL_2y \arrow[u, "G(\pi_2g^m)"']
\end{tikzcd}
        $$
        commute, where $g \coloneqq L(f)$.
        \item The composition of morphisms is defined through the composition in $C$.
    \end{itemize}
\end{definition}
Fix functors of variance $F\colon A\ra C$ and $G\colon B\ra C$. Let $(R,L)\colon A\ra B$ be a span. To see that the composition in the comma category $F\downarrow_L G$ is well defined, fix morphism $(x,\alpha)\xrightarrow{p} (y,\beta)\xrightarrow{q}(z,\gamma)$ in $F\downarrow_L G$. Denote $f = L(p)$ and $g = L(q)$. Now the diagram
$$
\begin{tikzcd}
                                                                          & F\pi_1 x \arrow[r, "\alpha"]            & G\pi_2 x \arrow[rd, "G\pi_2f^e"]                                                 &                                          \\
F\pi_1 f_s \arrow[ru, "F\pi_1f_m"] \arrow[rd, "F\pi_1f_e" description]    &                                         &                                                                                  & G\pi_2 f_t \arrow[d, "G\pi_2(g^ef^m)^e"] \\
F\pi_1(gf)_s \arrow[u, "F\pi_1(g_mf_e)_m"] \arrow[d, "F\pi_1(g_mf_e)_e"'] & F\pi_1 y \arrow[r, "\beta" description] & G\pi_2 y \arrow[ru, "G\pi_2f_m" description] \arrow[rd, "G\pi_2g^e" description] & G\pi_2(gf)_t                             \\
F\pi_1g_s \arrow[rd, "F\pi_1g_e"'] \arrow[ru, "F\pi_1g^m" description]    &                                         &                                                                                  & G\pi_2g_t \arrow[u, "G\pi_2(g^ef^m)^m"'] \\
                                                                          & F\pi_1 z \arrow[r, "\gamma"']           & G\pi_2 z \arrow[ru, "G\pi_2g_m"']                                                &                                         
\end{tikzcd}
$$
commutes by the functoriality of $F$ and $G$ and the definition of the morphisms $\alpha,\beta$ and $\gamma$. Thus $q\circ p$ defines a morphism $(x,\alpha)\ra (z,\gamma)$ in the comma category. 
\begin{example}
    The generalized comma category generalizes the notion of comma category, the category of algebras and coalgebras for contravariant and covariant functors. Interestingly, it contains the notion of heuristic naturality as well.
    \begin{itemize}
        \item Let $F\colon A\ra C$ and $G\colon B\ra C$ be functors with variance and let $(R,L)\colon A\ra B$ be a span. Consider the forgetful functor $U\colon F\downarrow_L G\ra R$. The $L$-transformations are in bijective correspondence with sections to $U$. An $L$-transformation $\alpha\colon F\Rightarrow G$ produces a section via $x\mapsto (x,\alpha_x)$ and $[f\colon x\ra y]\mapsto [f\colon (x,\alpha_x)\ra (y,\alpha_y)]$. The naturality condition of $\alpha$  concerning $f$ is exactly the condition for $\alpha(f)$ to be a well-defined morphism in the comma category.
        \item Let $F\colon A\ra C$ and $G\colon B\ra C$ be covariant functors. The usual comma category $F\downarrow G$ is the same as $F\downarrow_{Id} G$ where $Id\colon A\times B\ra A\times B$ is the identity functor.
        \item Let $F,G\colon A\ra B$ be a covariant functors. Reflect on the diagonal functor $\Delta\colon A\ra A\times A$ and the comma category
        $$
        F\downarrow_\Delta G.
        $$
        The objects of $F\downarrow_\Delta G$ could be called $F,G$-algebras. The objects consist of pairs $(a,t\colon F(a)\ra G(a))$ where $a$ is an object of $A$ and $t$ is a morphism in $B$. A morphism $f\colon (a,s)\ra (b,t)$ is a morphism $a\xrightarrow{f} b$ in $A$ making the diagram
        $$
        \begin{tikzcd}
F(a) \arrow[d, "s"'] \arrow[r, "F(f)"] & F(b) \arrow[d, "t"] \\
G(a) \arrow[r, "G(f)"']                & G(b)               
\end{tikzcd}
        $$
        commute. If $F$ or $G$ is the identity functor on $A$, then $F\downarrow_\Delta G$ is the category of $F$-algebras or the category $G$-coalgebras, respectively.

        Assume that $F$ and $G$ are the covariant power-set functor $\mathcal{P}\colon \textbf{Set}\ra \textbf{Pos}$ from the category of sets to the category of posets. Each topological space can be seen as a pair $(X,T)$, where $T\colon\mathcal{P}(X)\ra\mathcal{P}(X)$ is the closure operation, meaning it is an increasing map satisfying the conditions
        \begin{align*}
        &T(T(A)) = T(A)\\
        &A\subset T(A) \\
        &T(\emptyset) = \emptyset\\
        &T(A\cup B) = T(A)\cup T(B)
        \end{align*}

        for every $A,B\subset X$. The fixed points of $T$ define the family of closed sets of $X$. A continuous map $f\colon (X,S)\ra (Y,T)$ consists of a function $X\ra Y$ and $f(S(A))\subset T(f(A))$, which corresponds to $2$-categorical diagram
        $$
        \begin{tikzcd}
\mathcal{P}(X) \arrow[d, "S"'] \arrow[r, "\mathcal{P}(f)"]              & \mathcal{P}(Y) \arrow[d, "T"] \\
\mathcal{P}(X) \arrow[r, "\mathcal{P}(f)"'] \arrow[ru, "\leq", phantom] & \mathcal{P}(Y)               
\end{tikzcd}
        $$
If the inequality is equality, then the function $f$ is continuous and closed. Thus we see that the category of topological spaces with continuous closed maps is the full subcategory of the closure operations in $\mathcal{P}\downarrow_\Delta \mathcal{P}$.

\item Let $X$ be a set. Let $C$ be a locally small category. Denote with $\overline{X}$ the functor $1\ra \textbf{Set}$ from the terminal category to the category of sets choosing the set $X$. Examine the category $\overline{X}\downarrow_{1\times\Delta}\text{Hom}_C$. Here objects consists of pairs $(d,f\colon X\ra \text{Hom}_C(d,d))$. A morphism from $(c,f)$ to $(d,g)$ consists of a morphism $\phi\colon c\ra d$ making the diagram
$$
\begin{tikzcd}
                                   & {\text{Hom}_C(c,c)} \arrow[d, "\phi_*"]  \\
X \arrow[ru, "f"] \arrow[rd, "g"'] & {\text{Hom}_C(c,d)}                      \\
                                   & {\text{Hom}_C(d,d)} \arrow[u, "\phi^*"']
\end{tikzcd}
$$
commute. This makes $\phi$ an equivariant morphism. If $X$ is a monoid, then the full subcategory of monoid homomorphisms in $\overline{X}\downarrow_{1\times\Delta}\text{Hom}_C$ defines the functor category $[X,C]$, where the monoid $X$ is considered a single object category.
    \end{itemize}
    
\end{example}
\begin{theorem}[Componentwise naturality]\label{Componentwise naturality}
Let $F\colon C\ra D$ be a faithful functor. Let $L\colon B\ra D$ be a functor. Let $B'$ be a subgraph of $B$ generating $B$. Let $S\colon B'\ra D$ be a graph morphism. Assume that $F(S(f)) = L(f)$ for each morphism $f$ in $B'$. Then $S$ extends uniquely to a functor $S\colon B\ra D$. Furthermore, $FS = L$.
\end{theorem}

\begin{proof}
    The uniqueness of the extension is clear since $B'$ generates $B$. For existence, fix a morphism $f$ in $B$. Since $B'$ generates $B$, there exists a path $(f_n,\ldots, f_1)$ in $B'$ whose composite in $B$ is $f$. We define $S(f)\coloneqq S(f_n)\cdots S(f_1)$. By faithfulness of $F$, $S$ is a well-defined function on morphisms. The functoriality of $S$ follows from the faithfulness of $F$ and the functoriality of $L$. It is clear that $FS = L$.
\end{proof}

Consider a generalized comma category $F\downarrow_L G$ where $F\colon A\ra C$ and $G\colon B\ra C$ are functors with variance and $(R,L)\colon A\ra B$ is a span. The forgetful functor $U\colon F\downarrow_L G\ra R$ is faithful. As an immediate consequence of Theorem \ref{Componentwise naturality}, to check the heuristic naturality of a family of morphisms $FL_1x\ra GL_2x$ for objects $x$ in $R$, it suffices to check the naturality with respect to some class of morphisms that generate $R$. 

Let $I$ and $J$ be finite sets. Let $A$ be an $I$-list of categories and $B$ be a $J$-list of categories. Fix functors of variance $F\colon \prod A\ra C$ and $G\colon \prod B\ra C$. The variances need not be index-variances. Given a partition of arguments $R\colon A\ra B$ and a family $\eta$ of morphisms $\eta_x\colon F(\pi_1x)\ra G(\pi_2x)$ for $R$-objects $x$. To check the naturality of $\eta$, it suffices to check the naturality of $R$-morphisms $f$, where $f$ is supported only on a single equivalence class of $R$. This follows from Theorem \ref{Componentwise naturality}, since the morphisms that are identities up to an equivalence class of $R$ generate the rest. We may even allow the sets $I$ and $J$ to be infinite, but then the previous argument shows naturality with respect to morphisms supported on finitely many equivalence classes.

\section{Ends}

We define the notion of a wedge and a cowedge along a span on the domain of a functor of variance. As usual, a universal wedge defines an end and a universal cowedge defines a coend. The span on the domain functions as a choice along which the end is to be taken. The traditional definition is obtained when the span is chosen to be the diagonal $C\ra C\times C$ and the functor to be of index variance $(1,0)$ for a category $C$. We show that the usual formula for calculating ends via limits still holds as does the Fubini theorem. 
\begin{definition}[$L$-wedge]
    Let $F\colon A\ra C$ be a functor with variance. Let $(R,L)$ be a span on $A$. An $L$-wedge is a pair $(c,\eta\colon\overline{c}\Rightarrow F)$, where $c$ is an object of $C$ and $\eta$ is an $L$-transformation. A morphism $f$ between $L$-wedges $(c,\eta)$ and $(d,\theta)$ is a morphism $f\colon c\ra d$ making the diagram
    $$
    \begin{tikzcd}
c \arrow[dd, "f"'] \arrow[rd, "\eta_x"] &      \\
                                        & F(Lx) \\
d \arrow[ru, "\theta_{x}"']            &     
\end{tikzcd}
    $$
    commute for all objects $x$ of $R$. Via the composition in $C$, we attain the composition between morphisms of $L$-wedges. An $L$-\textbf{cowedge} is a pair $(c,F\xRightarrow[L]{} \overline{c})$. We define the categorical structure of $L$-cowedges similarly as compared to $L$-wedges.
\end{definition}

To be explicit, an $L$-wedge $(c,\eta)$ of a functor $F\colon A\ra C$ with variance consists of an object $c$ of $C$ and a collection $\eta$ of morphisms 
$$
(\eta_x\colon c\ra F(Lx))_{x\in R}
$$
such that any morphism $f\colon x\ra y$ in $R$ yields a commuting diagram
$$
\begin{tikzcd}
                                             & F(Lx) \arrow[d, "F(L(f)^e)"]  \\
c \arrow[ru, "\eta_x"] \arrow[rd, "\eta_y"'] & F(L(f)_t)                     \\
                                             & F(Ly) \arrow[u, "F(L(f)^m)"']
\end{tikzcd}
$$

\begin{definition}[Ends]
    Let $F\colon A\ra C$ be a functor with variance and let $(R,L)$ be a span on $A$. We call the terminal $L$-wedge of $F$ \textbf{the end of $F$ over $L$} and denote it by $(\int_L F,\omega\colon\overline{\int_L F}\xRightarrow[R]{}F)$ or by just $\int_L F$ and leave the \textbf{universal wedge} $\omega$ implicit. We denote by $\int^LF$ the initial $L$-cowedge of $F$ which is also called \textbf{the coend of $F$ over $L$}.
\end{definition}
It is useful to note now that any universal wedge will define a jointly monic family of morphisms.
\begin{theorem}\label{Ends as limits}
    Let $F\colon A\ra C$ be a functor with variance. Let $(R,L)$ be a span on $A$. Assume that the products
    $$
    \prod_x F(Lx) \text{ and }\prod_{f}F(L(f)_t)
    $$
    exist, where the indexing is over objects and morphisms of $R$, respectively. Consider the canonical morphisms
    $$
    s,t\colon \prod_x F(Lx)\ra \prod_{f}F(L(f)_t)
    $$
    defined by the diagram
    $$
   \begin{tikzcd}
                                                                                                     & F(x) \arrow[rd, "F(L(f)^e)"]                                                       &           \\
\prod_x F(x) \arrow[ru, "\pi_{x}"] \arrow[rd, "\pi_{y}"'] \arrow[r, "s"'] \arrow[r, "t", shift left] & \prod_{f}F(L(f)_t) \arrow[r, "\pi_f"] & F(L(f)_t) \\
                                                                                                     & F(y) \arrow[ru, "F(L(h)^m)"']                                                      &          
\end{tikzcd}
    $$
    where the upper part and the lower part commute for every morphism $f\colon x\ra y$ in $R$, respectively. Then the transposition 
    $$
    (c,\eta\colon \overline{c}\xRightarrow[R]{} F)\mapsto (c,\eta\colon c\ra\prod_x F(x))
    $$
    defines an isomorphism between the category of $L$-wedges of $F$ and the category of equalizing cones of $s$ and $t$. Especially, $\int_L F$ exists if and only the equalizer of $t$ and $s$ exists, and in such a situation they agree.
\end{theorem}
\begin{proof}
    It suffices to show that the transposition determines a well-defined bijection between the $L$-wedges and equalizing cones since the functoriality on morphisms is defined by the identity functor. 
    Let $\eta_x\colon c\ra F(Lx)$ be a morphism for every object $x$ of $R$. We show that the collection of morphisms $\eta = (\eta_x)_x$ defines an $L$-wedge if and only if the adjoint transposition $\eta'\colon c\ra \prod_x F(x)$ equalizes $s$ and $t$.
    Notice that $t\eta' = s\eta'$ is equivalent to $\pi_f t\eta' = \pi_fs\eta'$ for every morphism $f\colon x\ra y$ in $R$, which is equivalent with $F(L(f)^e)\eta_{x} = F(L(f)^m)\eta_{y}$ for every morphism $f\colon x\ra y$ in $R$. This is exactly the condition for $\eta$ to define an $L$-wedge.
\end{proof}

In Theorem \ref{Ends as limits} we may weaken the assumptions a bit. It suffices that the product 
$$
\prod_f F(L(f)_t)
$$
exists where indexing is over morphisms $f$ in a generating subgraph $R'$ of $R$. This is due to Theorem \ref{Componentwise naturality}, since to check naturality it suffices to check it only on morphisms in $R'$.

\begin{definition}
    Let $A$ and $B$ be categories with variance. Let $F\colon A\times B\ra C$ be a functor of variance with respect to the product variance on $A\times B$. We denote $F(x,y) = F_x(y) = F^y(x), F_x(g) = F(id_x,g)$ and $F^y(f) = F(f,id_y)$ for objects $(x,y)$ and morphisms $(f,g)$ the product category $A\times B$. This defines the functors of variance $F_x\colon B\ra C$ and $F^y\colon A\ra C$ for all objects $x$ in $A$ and $y$ in $B$.
\end{definition}
\begin{theorem}[Parameter theorem]
    Let $A$ and $B$ be categories with variance. Let $(R_1,L_1)$ and $(R_2,L_2)$ be spans on $A$ and $B$, respectively. Let $F\colon A\times B\ra C$ be a functor with variance. Assume that $[\int_{L_1} F](y)\coloneqq\int_{L_1} F^y$ exists with $\omega^y$ being the universal wedge for all objects $y$ in $B$. Then $\int_{L_1}F$ extends uniquely to a functor of variance $B\ra C$ so that the diagram
    $$
\begin{tikzcd}
{[\int_{L_1}F](g_s)} \arrow[d, "\int_{L_1}(F)(g)"'] \arrow[r, "\omega^{g_s}_a"] & {F(L_1(a),g_s)} \arrow[d, "{F(id,g)}"] \\
{[\int_{L_1}F](g_t)} \arrow[r, "\omega^{g_t}_a"']                               & {F(L_1(a),g_t)}                       
\end{tikzcd}   
    $$
    commutes for every object $a$ in $R$ and morphism $g$ in $B$.
\end{theorem}
\begin{proof}
Let $g$ be a morphism in $B$. First we show that $(F(id_{L_1(a)},g)\circ \omega^{g_s}_a)_{a\in R}$ defines an $L_1$-wedge $\overline{\int_{L_1} F^{g_s}}\Rightarrow F^{g_t}$. This then shows that $\int_{L_1} F$ is a well-defined function on morphisms. To show the $L_1$-naturality, fix a morphism $f\colon x'\ra y'$ in $R$. Write $L_1(f)$ as $h\colon x\ra y$. Consider the diagram
$$
\begin{tikzcd}[row sep = 1cm]
                                                                                    & {F(x,g_s)} \arrow[d, "{F(h^e,id_{g_s})}" description] \arrow[rr, "{F(id_x,g)}"]  &  & {F(x,g_t)} \arrow[d, "{F(h^e,id_{g_t})}"]  \\
\int_{L_1} F^{g_s} \arrow[ru, "\omega^{g_s}_{x'}"] \arrow[rd, "\omega^{g_s}_{y'}"'] & {F(h_t,g_s)} \arrow[rr, "{F(id_{h_t},g)}"]                                       &  & {F(h_t,g_t)}                               \\
                                                                                    & {F(y,g_s)} \arrow[u, "{F(h^m,id_{g_s})}" description] \arrow[rr, "{F(id_y,g)}"'] &  & {F(y,g_t)} \arrow[u, "{F(h^m,id_{g_t})}"']
\end{tikzcd}
$$
which commutes, since the left side commutes by the naturality of $\omega$ and the right side commute via the functoriality of $F$. Clearly $\int_{L_1}F\colon B\ra C$ maps identities to identities. To show mixed functoriality of $\int_{L_1} F$, fix morphisms $b\xrightarrow{f}b'\xrightarrow{g}b''$ in $B$. Consider the diagrams

\begin{equation}\label{Two diagrams}
\begin{tikzcd}[column sep = 1cm, row sep = 1 cm]
\int_{L_1}F(f_s) \arrow[rr, "\int_{L_1}F(f)"]                                                                                                                    &  & \int_{L_1}F(f_t) \arrow[d, "\int_{L_1}F((g^ef^m)^e)" description] & {F(L_1(a),f_s)} \arrow[rr, "{F(id,f)}"]                                                                                                         &  & {F(L_1(a),f_t)} \arrow[d, "{F(id,(g^ef^m)^e)}" description] \\
\int_{L_1}F((gf)_s) \arrow[d, "\int_{L_1}F(((g_mf_e)_e)" description] \arrow[rr, "\int_{L_1}F(gf)" description] \arrow[u, "\int_{L_1}F((g_mf_e)_m)" description] &  & \int_{L_1}F((gf)_t)                                               & {F(L_1(a),(gf)_s)} \arrow[u, "{F(id,(g_mf_e)_m)}" description] \arrow[d, "{F(id,(g_mf_e)_e)}" description] \arrow[rr, "{F(id,gf)}" description] &  & {F(L_1(a),(gf)_t)}                                          \\
\int_{L_1}F(g_s) \arrow[rr, "\int_{L_1}F(g)"']                                                                                                                   &  & \int_{L_1}F(g_t) \arrow[u, "\int_{L_1}F((g^ef^m)^m)" description] & {F(L_1(a),g_t)} \arrow[rr, "{F(id,g)}"']                                                                                                        &  & {F(L_1(a),g_t)} \arrow[u, "{F(id,(g^ef^m)^m)}" description]
\end{tikzcd}
\end{equation}
where the right diagram commutes. We show that the left diagram commutes as well. We may consider a cylindrical diagram whose top face is the left part and the bottom face is the right part of diagrams (\ref{Two diagrams}), when we add the universal wedges $\omega^b_a\colon \int_{L_1} F^b\ra F(L_1(a),b)$ for suitable objects $b$. The cylindric diagram is known to commute in all other faces except for the top face. This is because of the definition of $\int_{L_1}F$ and the fact that $F$ is a functor of mixed variance. Since the morphisms $\omega^b_a$ indexed over objects $a$ in $A$ are jointly monic, for $b = (gf)_t$, it follows from a diagram chase that the top face commutes as well. 
\end{proof}

\begin{theorem}[Limits commute, Fubini]
    Let $A$ and $B$ be categories with variance. Let $(R_1,L_1)$ and $(R_2,L_2)$ be spans on $A$ and $B$, respectively. Let $F\colon A\times B\ra C$ be a functor of variance with respect to the product variance on $A\times B$. Denote $L = L_1\times L_2\colon R\ra A\times B$ Assume that the end
    $$
    \int_{L_1}F^b
    $$
    exists with the universal wedge being $\omega^b\colon \overline{\int_{L_1}F^b}\Rightarrow F^b$ for every object $b$ of $B$. Consider the corresponding canonical functor $\int_{L_1} F\colon B\ra C$. Then the transposition
    $$
    (c,\theta\colon \overline{c}\xRightarrow[L_2]{}\int_{L_1} F)\mapsto (c,\Tilde{\theta}\colon \overline{c}\xRightarrow[L]{} F),
    $$
    where $\tilde{\theta}_{x,y} = \omega^{L_2(y)}_x\circ \theta_y$ for objects $(x,y)$ in $R$, defines an isomorphism between the categories of $L$-wedges of $F$ and the $L_2$-wedges of $\int_{L_1} F$. In particular, the end $\int_L F$ exists if and only if the end $\int_{L_2}\int_{L_1}F$ exists, and if one exists, they agree.
\end{theorem}

\begin{proof}
    First, we show that the transposition is well-defined. Let $(c,\theta)$ be an $L_2$-wedge of $\int_{L_1} F$. We show that $\tilde{\theta}$, which is the collection of morphisms
    $$
    \tilde{\theta}_{x,y}\colon c\xrightarrow{\theta_y} \int_{L_1} F^{L_2(y)}\xrightarrow{\omega^{L_2(y)}_x} F(L(x,y)),\text{ for objects $(x,y)$ in $R$},
    $$
    defines an $L$-wedge. Let $(p,q)\colon (a,x)\ra (b,y)$ be a morphism in $R$. Denote $(f,g)\coloneqq L(p,q)$. Consider the commutative diagram
    $$
    \begin{tikzcd}[column sep = 2 cm, row sep = 1cm]
c \arrow[r, "\theta_x"] \arrow[d, "\theta_y"']                                      & \int_{L_1}F^{L_2(x)} \arrow[r, "\omega^{L_2(x)}_a"] \arrow[d, "\int_{L_1}F(g^e)" description]     & {FL(a,x)} \arrow[d, "{F(id, g^e)}"]      \\
\int_{L_1}F^{L_2(y)} \arrow[d, "\omega^{L_2(y)}_b"'] \arrow[r, "\int_{L_1}F(g^m)"'] & \int_{L_1}F^{g_t} \arrow[d, "\omega^{g_t}_c" description] \arrow[r, "\omega^{g_t}_a" description] & {F(L_1(a),g_t)} \arrow[d, "{F(f^e,id)}"] \\
{FL(b,y)} \arrow[r, "{F(id,g^m)}"']                                                 & {F(L_1(b),g_t)} \arrow[r, "{F(f^m,id)}"']                                                         & {F((f,g)_t)}                            
\end{tikzcd}
    $$
The rectangles commute, because of the $L_2$-naturality of $\theta$, the definition of the functor $\int_{L_1}F$ and the $L_1$-naturality of $\omega^{g_t}$.

To see the bijectivity, fix an $L$-wedge $(c,\eta)$ of $F$. We construct an $L_2$-wedge $(c,\theta)$ of $\int_{L_1}F$. We define $\theta_y\colon c\ra [\int_{L_1}F](L_2(y))$ to be the unique morphism making the diagram
$$
\begin{tikzcd}
                                                      & \int_{L_1}F^{L_2(y)} \arrow[d, "\omega^{L_2(y)}_x"] \\
c \arrow[r, "{\eta_{(x,y)}}"'] \arrow[ru, "\theta_y"] & {FL(x,y)}                                          
\end{tikzcd}
$$
commute for each object $x$ in $R_1$. Fix an object $y$ in $R_2$. The morphisms $\theta_y$ are well defined, since $L$ naturality of $\eta$ implies that $(\eta_{x,y})_x$ is $L_1$-natural for each $y$. To see that the family $(\theta_y)_y$ is $L_2$-natural, fix a morphism $p\colon b\ra y$ in $R_2$. Denote $g = L_2(p)$. Let $a$ be an object of $R_1$. Consider the diagram
$$
\begin{tikzcd}[row sep= 1cm]                                                                                                                                        & {F(L_1(a),L_2(b))} \arrow[rdd, "{F(id,g^e)}"]                                                             &                 \\
                                                                                                                                                        & \int_{L_1}F^{L_2(b)} \arrow[d, "\int_{L_1}F(g^e)" description] \arrow[u, "\omega_a^{L_2(b)}" description] &                 \\
c \arrow[ru, "\theta_b" description] \arrow[ruu, "{\eta_{a,b}}", bend left] \arrow[rdd, "{\eta_{a,y}}"', bend right] \arrow[rd, "\theta_y" description] & \int_{L_1}F^{g_t} \arrow[r, "\omega^{g_t}_a"]                                                             & {F(L_1(a),g_t)} \\
                                                                                                                                                        & \int_{L_1}F^{L_2(y)} \arrow[u, "\int_{L_1}F(g^m)" description] \arrow[d, "\omega_a^{L_2(y)}" description] &                 \\
                                                                                                                                                        & {F(L_1(a),L_2(y))} \arrow[ruu, "{F(id,g^m)}"']                                                            &                
\end{tikzcd}
$$
A diagram chase using the fact that the family $(\omega_a^{g_t})_a$ of morphisms is jointly monic shows that the family of morphisms $(\theta_y)_y$ is $L_2$-natural at $p$. This shows that the transposition defined in the theorem is bijective. The mapping of morphisms $(c,\theta)\xrightarrow{\alpha} (d,\phi)\mapsto (c,\tilde{\theta})\xrightarrow{\alpha}(c,\tilde{\phi})$ is well defined and extends the transposition into an isomorphism of categories.
\end{proof}

\section*{Acknowledgements}
I would like to express a heartfelt gratitude to my supervisor, Professor Van der Linden, for their insightful feedback, and expert guidance throughout preparing my first article. I am grateful for their mentorship.

I extend my appreciation to my colleagues in the Junior Category Theory Seminar hosted in UCLouvain for the inspiration.

\bibliographystyle{unsrtnat}
\bibliography{tim}
\end{document}